\documentclass[10pt]{article} 
\usepackage[preprint]{tmlr}

\usepackage{amsfonts,amsmath,latexsym,amssymb,verbatim,amsbsy,amsthm}
\usepackage{dsfont}

\usepackage{graphicx}

\usepackage{caption}
\usepackage{float}
\usepackage{wrapfig}

\usepackage{hyperref}
\usepackage{url}
\usepackage{cleveref}

\def \d {\delta}
\def \k {\kappa}
\def \e {\varepsilon}

\def \l {\lambda}

\def \s {\sigma}

\def \L {\Lambda}

\def \bu {{\bf u}}

\def \bx {{\boldsymbol x}}
\def \by {{\boldsymbol y}}

\def \cW {\mathcal{W}}

\newcommand{\R}{\ensuremath{\mathbb{R}}}   

\def \ds  {\, \mbox{d}s}

\def \ddt  {\frac{\mbox{d\,\,}}{\mbox{d}t}}

  \newtheorem{theorem}{Theorem}[section]
  \newtheorem{lemma}[theorem]{Lemma}

  \newtheorem{definition}[theorem]{Definition}

\title{Unique Nash equilibrium of a nonlinear model of opinion dynamics on networks with friction-inspired stubbornness}

\author{\name David N. Reynolds \email david.reynolds@gssi.it \\
      \addr School of Mathematics,\\
      Gran Sasso Science Institute \\
      67100 L'Aquila, Italy
      \AND
      \name Francesco Tudisco \email f.tudisco@ed.ac.uk \\
      \addr School of Mathematics and Maxwell Institute for Mathematical Sciences\\
      University of Edinburgh, EH93FD Edinburgh, UK\\ Gran Sasso Science Institute, 67100 L'Aquila, Italy}

\begin{document}

\maketitle

\begin{abstract}
  The modeling of opinion dynamics has seen much study in varying academic disciplines. Understanding the complex ways information can be disseminated is a complicated problem for mathematicians as well as social scientists. We present a nonlinear model of opinion dynamics that utilizes an environmental averaging protocol similar to the DeGroot and Freidkin-Johnsen models. Indeed, the way opinions evolve is complex and nonlinear effects ought to be considered when modelling. For this model, the nonlinearity destroys the translation invariance of the equations, as well as the convexity of the associated payout functions. The standard theory for well-posedness and convergence no longer applies and we must utilize the Brouwer topological degree and nonconvex analysis in order to achieve these results. Numerical simulations of the model reveal that the nonlinearity behaves similarly to the well-known Friedkin-Johnsen for so-called ``reasonable'' opinions, but better models the way agents that hold ``extreme'' opinions are more stubborn than their reasonable counterparts.
\end{abstract}

\paragraph{Keywords}
  Graphs, networks, continuous-valued opinion dynamics, consensus, compromise, stubbornness

\paragraph{AMS subject classification}
  91D30, 
  05C57, 
  05C50, 
  34A34, 
  34D05 

\thispagestyle{plain}
\section{Introduction}

Humanity constantly exchanges information amongst itself; from watching the news and reading articles online, to speaking with acquaintances or simply overhearing a conversation at a bar. The means by which information is transmitted are innumerable. From all this input and our inherent predispositions, we are constantly forming and adjusting our opinions about any number of topics. With this comes a desire to understand and model the way opinions change. In recent years there has been an increase in interest in this exact topic in many academic areas, but in particular, there is much mathematics being done to rigorously describe and predict how opinions develop within a group of agents.

Networks represented as undirected graphs $\mathcal G = (\mathcal V,\mathcal E)$ with $N$ nodes $\{1,\dots,N\}$ and undirected edges $\mathcal E\subseteq \mathcal V\times \mathcal V$ are one of the most successful models for opinion dynamics. A wealth of real-world social and natural interaction scenarios are very well represented via a network where each agent is represented by a vertex $i\in \mathcal V$ and two agents $i,j\in \mathcal V$ exchanging information directly are shown via the edge $ij\in \mathcal E$ that connects them \cite{friedkin2011social,newman2018networks}. Within this framework, many models have been created to describe certain phenomena like consensus, disagreement, spread of diseases, and even biological models of flocking dynamics \cite{CS2007a,CS2007b,DeGroot,FJ,Heg,JMFB,LRS}.  
In the context of opinion dynamics, the majority of the models  build on a wealth of empirical evidence 
which shows that individuals update their opinions as  combinations of their own and others' 
opinions, based on the strength of the reciprocal interpersonal ties \cite{ALT,DeGroot,French,Harary2, Harary1}. This is represented in the celebrated DeGroot model by means of a linear dynamical system where the opinion $x_i$ of node $i$ evolves as a function of its neighboring nodes in the graph: $x_i(k+1)=\sum_{j: ij\in \mathcal E}x_j(k)$. For any aperiodic graph, this model always reaches consensus, in that all the agents (the nodes) will eventually share the same opinion. In fact, it is straightforward to show that in that case a unique limit $\lim_{k\to \infty}x_i(k)=x_i^*$ exists and $x_i^*=x_j^*$ for all $i,j$. The Friedkin-Johnsen model further adds a linear term to model stubbornness, i.e.\ the propensity (or resiliency) of each individual to change their own initial convictions. With this addition, consensus may not happen, however, due to the linearity (and thus the translation invariance) of this model, the opinions of the network tend to become ``similar'' even when the initial convictions and levels of stubbornness are very different. In order to better model opinion dynamics it is key to be able to incorporate nonlinear effects into the model. 

In \cite{XYLCS} the authors treat a nonlinear version of the Degroot model taking into account the psychological phenomena of confirmation bias.  Further, \cite{YCC} investigates incorporating a nonlinear force into the continuous-time Degroot model. In both \cite{XYLCS}, and \cite{YCC}, the asymptotic state is, in general, consensus.  We introduce a nonlinear variation of the Friedkin-Johnsen model, which better models the desire of the agents to maintain their own convictions rather than moving towards a global consensus, i.e. nonlinear stubbornness. The particular nonlinear term is inspired by the Rayleigh's self-propulsion and friction force. Indeed, this forcing term has seen study in the context of flocking and swarming models of collective dynamics \cite{LRS}. Understanding the way this individual forcing term interacts with the consensus operator is also directly connected to the behaviour of synchronous systems like the Stuart-Landau oscillators \cite{PLA}. Our formulation in terms of a differential opinion game allows us to prove the existence of a unique Nash equilibrium to which the system converges along a gradient flow, \cite{Li-Xue-Yu,Simon}. Moreover, it can be applied to previous linear models as well and allows us to revisit their convergence analysis in terms of Nash equilibria.

The rest of the paper is structured as follows. First, in \Cref{Pre} we introduce our notation and preliminary considerations. Then, the remainder of this section is devoted to overviewing relevant background, in particular the Friedkin-Johnsen model, and introducing the nonlinear generalization we consider in this work.   \Cref{warmup} is a standard analysis of the continuous DeGroot model as a warm-up. \Cref{MR} will be dedicated to proving the main result of the paper, \Cref{T:main}. In \Cref{s:CFJ} we will apply the technique developed in the preceding section to the continuous version of the Friedkin-Johnsen model. Finally, in  \Cref{sec:experiments} we will present a few numerical examples to show the different behavior the nonlinear opinion dynamical term introduces.

\subsection{Preliminaries}\label{Pre}
In this paper, each model will represent how the opinions of a group of agents, on a particular topic, evolve with time. Therefore we will be considering opinion values $x_i(k),x_i(t) \in \R_+$ where $k$ denotes discrete time steps,  $t$ the continuous time evolution, and $i$ the specific agent. For interpretation as an opinion model we consider small values close to $0$ representing disagreement about a particular topic, and as values grow larger the more the agent express agreement, where for large values $x_i\gg 1$, the held opinion would be considered ``extreme''. Similarly, each agent will be allowed a desired or initial opinion value $u_i$, which from now on will be referred to as \textit{convictions} and will also reside in the positive reals $u_i \in \R_+$. The graph $\mathcal{G}$ of the network will be connected and undirected unless otherwise stated, with node set $\{1,\dots,N\}$ representing the group of agents being modeled. Finally, while the models discussed in the  introduction will be based on the normalized adjacency matrix $A$  as defined in \eqref{eq:normalized_adj_mx}, this choice of linear mapping is not a necessity for the main analysis and in later sections we will broaden the scope to  any symmetric entrywise nonnegative $N\times N$ matrix $M$. We will specify $M$ in specific application contexts.

\subsection{DeGroot and Friedkin-Johnsen models}
Opinions can be tracked as  binary values (agreement or disagreement); as in the famous Voter Model of Holley and Liggett \cite{HL} or individual words as in the Naming Game seen in \cite{MPH, XSKZLS}; or they can be continuous values where a larger value could correspond with agreement and lower with disagreement about a particular topic, \cite{Biz,BCP, DeGroot, French, FJ, Heg, LRS}. The method by which the dynamics themselves occur can be described through Bayesian \cite{MPH} or Non-Bayesian processes \cite{DeGroot, FJ}. Further, there are Bounded Confidence models that take into account with which other agents interactions can occur (e.g.\ only those with similar opinions \cite{Def, Heg}), and models that incorporate the stubbornness of an agent, leading to disagreement, or compromises instead of consensus \cite{FJ, LRS, PTCF, XZGZZ}. See \cite{PT1, PT2} for an overview of the various mathematical models that have been developed in the field of opinion dynamics.

The early work of French and Harary \cite{French, Harary2} on reaching consensus via communication on a graph, set the stage for what is now known as DeGroot learning \cite{DeGroot}; a Non-Bayesian process by which a consensus is achieved via updating beliefs at each time step, based off of interactions with others in the network. Indeed, the standard model for achieving consensus is for each agent $i$ in a network to update its current beliefs $x_i(k)$ at time step $k$ via an environmental averaging over the beliefs of the agents with which it shares a connection:
\begin{equation}\label{eq:EA}
x_i(k+1)=\frac{1}{\deg(i)}\sum_{j: ij\in \mathcal E}x_j(k)= \sum_{j=1}^NA_{ij}x_j(k),
\end{equation}
where $\deg(i) = |\{j:ij\in \mathcal E\}|$ is the degree of node $i$, and $A_{ij}$ are the elements of the row-stochastic ``normalized'' adjacency matrix
\begin{equation}\label{eq:normalized_adj_mx}
    A_{ij}=\begin{cases}
\deg(i)^{-1} & i\sim j\\
0 & \text{otherwise}
\end{cases} \, .
\end{equation}
 
Indeed, this environmental averaging protocol is also seen in biological flocking dynamics \cite{CS2007a, CS2007b, LRS, Rbook}. This protocol can be derived from each agent attempting to minimize the payout function
\begin{equation}\label{eq:pEA}
    p_i(\bx)=\frac{1}{2}\sum_{j=1}^NA_{ij}(x_i-x_j)^2.
\end{equation}
 It is well known that, so long as the adjacency matrix is not periodic, this model produces convergence of opinions to a consensus value which is a convex combination of the initial opinion values.

The Friedkin-Johnsen model \cite{FJ, GS, PTCF} uses this mechanic to spread information between agents, but further includes a stubbornness factor $\lambda_i$ as a means of ``anchoring'' opinions to a particular fixed value,

\begin{equation}\label{eq:FJ}
    x_i(k+1)=\l_i\sum_{j=1}^N A_{ij}x_j(k)+(1-\l_i)u_i, \ \ \l_i \in [0,1].
\end{equation}
In this way the opinions $x_i$ are torn between the desired consensus value and its own fixed preferred value $u_i$, often taken to be the initial opinion value $x_i(0)$. Each step is taken from a weighted average of these two values that can similarly be chosen via minimizing the payout function

\begin{equation}\label{eq:pFJ}
    p_i(\bx)=\frac{\l_i}{2}\sum_{j=1}^NA_{ij}(x_i-x_j)^2+\frac{1-\l_i}{2}(x_i-u_i)^2.
\end{equation}

If all $\l_i=1$, then none of the agents are stubborn and we return to consensus as in \eqref{eq:EA}. Alternatively, if $\l_i=0$ for all $i$, then that agent is fully stubborn and its opinion will be exactly $x_i=u_i$, regardless of any other opinion values. Finally, for an undirected connected graph $\mathcal G$ with at least one $0<\l_i<1$, it is well known that the Perron-Frobenius theory entails this model has a unique equilibrium vector $\bx^*$ to which the opinion values will eventually converge. In general, this equilibrium will not represent consensus, but rather what we call \textit{a compromise}. Note that, due to the linearity of the system, the stability of the fixed point can be immediately seen as the linearization of the system is represented by a stable diagonal shift of the graph Laplacian.

\subsection{A nonlinear generalization of the FJ model with friction inspired stubbornness}

Note that we can rewrite the FJ protocol \eqref{eq:FJ} by letting $y_i=\frac{x_i}{\l_i}$ and $\sigma_i=\frac{1-\l_i}{\l_i}$ as 
\begin{equation}\label{eq:FJs}
    y_i(k+1)=\sum_{j=1}^NB_{ij}y_j(k)+\s_iu_i, \ \ \s_i \in [0,\infty],
\end{equation}
where $B_{ij}=\l_jA_{ij}$ is a substochastic matrix. This form leads nicely to the protocol we introduce here. 

The new protocol we introduce is similar to the FJ model, except it introduces a nonlinearity into the stubbornness term which substantially complicates the analysis of the model. This protocol has been seen in the Cucker-Smale system of flocking dynamics, with a Rayleigh friction and self-propulsion forcing term driving the dynamics \cite{LRS,NR}.  However,  while in the flocking case, the Cucker-Smale protocol is a time-dependent all-to-all coupling, in our settings all couplings are over a fixed network that allows for more complicated topologies. Interpreting  stubbornness within opinion dynamics as a form of friction yields the following update protocol with a nonlinear stubbornness forcing term:

\begin{equation}\label{eq:NFJ}
    x_i(k+1)=\sum_{j=1}^N M_{ij}x_j(k)+\s_i(u_i-x_i^p(k))x_i(k),
\end{equation}
where $M$ is any entrywise nonnegative graph-related matrix, $p> 0$ is a parameter responsible for the ``strength'' of the nonlinearity and we use $\sigma_i \geq 0$ to represent the level of an agent's desire to believe $u_i^{1/p}$. For $M=B$, the way the update protocol works is similar to the FJ model, but rather than averaging between the current average opinion and a fixed value, the agent observes the current average belief of those it is connected to, $\sum_{j=1}^NM_{ij}x_j$, as well as how far away their current belief, $x_i(k)$, is from their desired belief $u_i^{1/p}$. If the agent's current belief is close to its desired belief, or the strength measured by  $\sigma_i$ is small, then the update will be closer to the average belief. However, if the current belief is far from the desired belief, it will move away from the average, towards the desired belief. In this way, the model balances the agents' desire to agree with the average value and its desire to believe the specific value $u_i^{1/p}$ in a nonlinear fashion. In other words, \eqref{eq:NFJ} models agents that are willing to compromise, but the extent to which is determined by the parameters, and strength of nonlinearity.

Similar to the FJ model, the nonlinear protocol in \eqref{eq:NFJ} also has an associated payout function:

\begin{equation}\label{eq:pNFJ}
    p_i(\bx)=\frac{1}{2}\sum_{j=1}^N M_{ij}(x_i-x_j)^2+\sigma_i\left(\frac{1}{p+2}x_i^{p+2}-\frac{1}{2}u_ix_i^2\right),
\end{equation}
where the update protocol \eqref{eq:NFJ} corresponds with a stationary point of this payout function. Note that, the lack of convexity in \eqref{eq:pNFJ} means that we do not know a-priori that the system has a unique minimizer for \eqref{eq:pNFJ}. Therefore, \eqref{eq:NFJ} is not the only update protocol one can associate with that payout function. Alternative choices include constant multiples of \eqref{eq:NFJ} as well as the following protocol
\begin{align}\label{eq:bNFJ}
    x_i(k+1)=\left(\sum_{j=1}^N M_{ij}x_j(k)\right)\left(1-\sigma_i(u_i-x_i^p(k))\right)^{-1}.
\end{align}

Heuristically, \eqref{eq:bNFJ}
would achieve the same goal of balancing between the average value and the desired value $u_i^{1/p}$, but it raises worries if $(1-\sigma_i(u_i-x_i^p(k)))=0.$ For this reason, \eqref{eq:NFJ} will be the nonlinear update protocol, associated with the payout function \eqref{eq:pNFJ}, to which we will refer.

In this framework, the dynamics are considered to take place at discrete time steps, and indeed for practical applications and numerical computations this is an ideal framework. However, if we allow for opinions to evolve continuously in time, we can re-frame each of these models as a system of Ordinary Differential Equations (ODEs). To this end, one possibility is to subtract $x_i(k)$ on both sides of the discrete-time models and formally replace the discrete derivative $x_i(k+1)-x_i(k)$ with $\frac {\mathrm d}{\mathrm d t} x_i(t) = \dot x_i(t)$, where we now use the symbol $t$ to emphasize that we let $t$ vary continuously. In this way, we observe that the continuous-time equivalent of the Degroot model \eqref{eq:EA} corresponds to the system of ODEs for environmental averaging: 
\begin{equation}\label{eq:EAc}
    \dot{x}_i(t)=\sum_{j=1}^NA_{ij}(x_j(t)-x_i(t)),
\end{equation}
with identical payout function \eqref{eq:pEA}. Indeed, the continuous analogue has its own name as the Abelson Model \cite{PT1,PT2}.  Similarly, for \eqref{eq:FJ} we arrive at:
\begin{equation}\label{eq:FJc}
    \dot{x}_i(t)=\l_i\sum_{j=1}^NA_{ij}(x_j(t)-x_i(t))+(1-\l_i)(u_i-x_i(t)),
\end{equation}
obtaining the differential analog of the FJ model with the same corresponding payout function \eqref{eq:pFJ}. This model is known as the Taylor Model \cite{PT1,PT2}. Last, the continuous version of the nonlinear model defined in \eqref{eq:NFJ} is given by
\begin{equation}\label{eq:NFJc}
    \dot{x}_i(t)=\sum_{j=1}^N M_{ij}(x_j(t)-x_i(t))+\s_i(u_i-x_i(t)^p)x_i(t).
\end{equation}

We note that in this continuous setting, the payout--protocol correspondence may change. In fact, while this model corresponds the same payout function \eqref{eq:pNFJ}, using a different protocol here, as for instance \eqref{eq:bNFJ}, would lead to a different differential analogue, and would have a different corresponding payout function. This is one more reason why we prefer \eqref{eq:NFJ} over e.g. \eqref{eq:bNFJ}, as we consider the `correct' continuous analogue of each system to be the one that seeks stationary points to the same payout function. In this way,  each of these systems can now be viewed as a non-cooperative dynamical opinion game, where agents continuously update their opinions to minimize the corresponding payout function.

Throughout the rest of the paper, we will analyze these models in the framework of differential opinion games, proving that each of them has a Nash equilibrium to which the system converges. In the case of \eqref{eq:EAc}, the equilibrium will be the consensus of opinions, but the diagonal forcing included in \eqref{eq:FJc} and \eqref{eq:NFJc} makes consensus impossible, leading to the equilibrium opinion vector being representative of a best-case compromise.

\subsection{Game theoretic set-up and statement of main results} We will consider each of \eqref{eq:EAc}, \eqref{eq:FJc}, and \eqref{eq:NFJc} as a noncooperative dynamical opinion game, where the goal of each agent is to minimize their respective payout functions \eqref{eq:pEA}, \eqref{eq:pFJ}, and \eqref{eq:pNFJ}. In this way we will show that the dynamics of each of these systems lead to a steady state which is an exponentially stable Nash equilibrium. For the Abelson and Taylor models these results are already known, although they are not usually framed in this context, while for the nonlinear model presented here we must employ different techniques due to the instability of the diagonal forcing term.

\begin{definition}[Nash Equilibrium]\label{Nas} 
An $N$-tuple of strategies, $\bx=(x_1,...,x_N)$, is a Nash Equilibria if and only if for each $i$, the payoff function
\begin{align}
p_i(\bx)=\max_{r_i}p_i(x_1,x_2,...,r_i,...,x_N),
\end{align}
i.e., no player can unilaterally increase their payout given the strategy of each other agent is fixed.
\end{definition}

For background on Nash equilibria see \cite{Nash}. The main result that is to be shown in this paper is the following.

\begin{theorem}\label{T:main}
For any set of parameters $(\bu,M,\boldsymbol{\s})$, such that $u_i \in \R_+$, $M_{ij}=M_{ji}\geq 0$, and $\s_i\geq 0$, there exists a unique solution $\bx^*\in \R^N_+$ to \eqref{steady} which is a locally exponentially stable equilibrium for the system \eqref{model}. The map $(\bu,M,\boldsymbol{\s})\mapsto \bx^*$ 
is infinitely smooth. Moreover, any solution $\bx(t)\in \R^N_+$ converges to the unique Nash equilibrium.
\end{theorem}

Before proving this result we present a standard analysis of the convergence of the Abelson model.


\section{Continuous Degroot Learning: The Abelson Model}\label{warmup}
In this section, we will do a quick analysis of model \eqref{eq:EAc}, which we recall below
\begin{align*}
    \dot{x}_i&=\sum_{j=1}^N A_{ij}(x_j-x_i), \ \ \ i=1,...,N,\\
    x_{i0}&=x_i(0), \ \ i=1,..,N.
\end{align*}
Let the average opinion value be denoted by $\bar{x}=\frac{1}{N}\sum_{j=1}^Nx_j$. As $A$ is symmetric, this average opinion is conserved in time,
\begin{align}
    \ddt \bar{x}=\frac{1}{N}\sum_{i,j=1}^NA_{ij}(x_j-x_i)=0
\end{align}

Denote $\max_i x_i(t)=x_+(t)$ and $\min_i x_i(t)=x_-(t)$. Then, the continuous Degroot learning system obeys the maximum and minimum principles. Indeed, as $x_+(t), x_-(t)$ are Lipschitz continuous, we can differentiate almost everywhere to get,

\begin{align*}
  \ddt x_+(t)&=  \sum_{j=1}^N A_{+j}(x_j-x_+)\leq 0,\\
  \ddt x_-(t)&= \sum_{j=1}^N A_{-j}(x_j-x_-)\geq 0,
\end{align*}
where $A_{+j}$ and $A_{-j}$ are the coefficients associated with the edge connecting nodes $x_+(t)$ (or $x_-(t)$) and $x_j$.

It is easy to see that any consensus, i.e.\ any constant vector of opinions $\bx=(x,...,x)$, represents a steady state of the system. As $\sum_j A_{ij}(x_j-x_i)=- \mathbf 1_i^\top (D_A-A)\bx$, with $\mathbf 1_i$ the $i$-th canonical vector and $D_A$ the diagonal matrix with $i$-th diagonal entry $\sum_jA_{ij}$,  we see that the first eigenvalue $\kappa_1$ of $D_A-A$ is zero, $\kappa_1=0$. Further, as $A$ is connected, the second eigenvalue of $D_A-A$, the Fiedler number, is positive, $\kappa_2>0$, which means that consensus is the only possible steady state of the system. As the average is conserved, we now see that the dynamics must lead to consensus at the average value $\bx=(\bar{x},...,\bar{x})$. 
Indeed, if the graph is connected, then we can prove an exponential rate of convergence. From the variational characterization $\kappa_2 = \min_{\bx,\by}\langle (D_A-A)\bx,\by\rangle/\|\bx\|\|\by\|$  and the Cauchy-Schwarz inequality, we get 
\[
\langle (D_A-A)\bx,\by\rangle\geq \kappa_2 \langle \bx,\by\rangle\,
\]
for any two vectors $\bx,\by$. Thus, 
\begin{align*}
    \ddt(x_+-x_-)(t)&=\sum_{j=1}^NA_{+j}(x_j-x_+)-A_{-j}(x_j-x_-),\\
    &= -\langle (D_A-A)\bx,\mathbf{1}_+-\mathbf 1_-\rangle \leq  -\kappa_2 (x_+-x_-),
\end{align*}
yielding $(x_+-x_-)(t)\leq (x_+-x_-)(0)e^{-\kappa_2t}$. An important note for the Abelson model is that it reaches consensus for any entrywise nonnegative matrix, while in the original Degroot model, periodic matrices will not necessarily yield consensus, but rather a periodic behavior due to the synchronous switching of opinion values at each time step. Further, if the matrix $A$ is not symmetric, then the average opinion value will not be conserved, and the limiting consensus value will be an emergent convex combination of the initial values.

In general, consensus is not a realistic result. Introducing stubbornness to the system leads to the possibility of a different equilibrium where not all agents agree, as is more often seen in reality. Incorporating stubbornness leads to a diagonal shift of the graph Laplacian, in the case of the Taylor model \eqref{eq:FJc}, this shift is positive and linear and therefore stable. This results in the first eigenvalue, $\k_1$, of $D_A-A$ shifting from zero to being strictly positive, reducing the family of steady states that is consensus in the Abelson model, to a unique equilibrium in the Taylor model. The difficulty in analyzing our particular nonlinear forcing is that the diagonal shift is not stable. Indeed, in general, the Jacobian at the fixed point will not be diagonally dominant as in the case of the Taylor model. Therefore, we cannot immediately deduce a unique equilibrium value as we cannot a priori know where the eigenvalues will shift. In the following section, we prove first that the dynamics remain bounded away from the trivial solution at $\bx=0$. Using this fact we can determine that the Jacobian at any positive equilibrium is not critical. However, due to the lack of convexity within the payout functions and hence the gradient flow, we utilize the Brouwer topological degree to achieve uniqueness of the steady state. The main theorem will then be proved after using the Lojasiewicz gradient inequality to achieve convergence for all positive initial data.

\section{Main results: Nonlinear Opinion Dynamics}\label{MR}
In this section, we will prove the main result of the paper, Theorem \ref{T:main}, showing existence of a unique Nash Equilibrium for the model \eqref{eq:NFJc}. Let us rewrite the system here for reference,
\begin{align}\label{model}
 \dot{x}_i&=\sum_{j=1}^N M_{ij}(x_j-x_i)+\sigma_i(u_i-x_i^p)x_i, \ \ \s_i\in[0,\infty],\\
 x_{i0}&=x_i(0), \ \ i=1,...,N,
\end{align}
where $M$ is any symmetric entrywise nonnegative matrix. Then this can be viewed as a model of opinion dynamics where $x_i(t)$ tracks opinion values, $u_i$ represents a fixed conviction value, and each $\sigma_i$ is a fixed measure of each agent's stubbornness. The first half of the model acts as averaging amongst the other agents to bring opinions towards a consensus value. While the second half, derived from the Rayleigh's friction and self-propulsion in \cite{LRS}, attempts to push each opinion toward the conviction value $u_i^{1/p}$. Note that the stubbornness parameters, $\s_i$, are allowed to take the values zero or infinity. When $\s_i=0$, the agent behaves exactly as an agent in the Abelson model, converging to the average value of its connections, while for $\s_i=\infty$, the agent is perfectly stubborn and remains fixed at value $u_i^{1/p}$. These terms represent agents that are completely flexible (not stubborn) in the case of $\s_i=0$, and in the case of $\s_i=\infty$, can be seen as a source term like a particular newspaper or journal. As the dynamics of these cases are already known  we need only consider the case of $0<\s_i<\infty$ for all $i=1,...,N.$


\subsection{Boundedness of the dynamics}
First, let us prove that all opinion values satisfy a type of Maximum Principle.
\begin{lemma}[Minimum/Maximum Principle]\label{mmp}
For initial conditions $x_{i0}\in \R_+$, for each $i=1,...,N$, the solutions to \eqref{model} remain bounded above and below for all time, i.e.,
\begin{align}\label{mmpineq}
    \min_{i}\{\min(x_i^p(0),u_i)\}=c_-\leq x_i^p(t)\leq c_+= \max_{i}\{\max(x_i^p(0),u_i)\}
\end{align}
for all $t\geq0$.
\end{lemma}
\begin{proof}
Initially \eqref{mmpineq} is satisfied. Let $x_-(t)=\min_i x_i(t)$,  then differentiating,
\begin{align*}
\ddt x_-&=\sum_{j=1}^N M_{-j}(x_j-x_-)+\sigma_-(u_--x_-^p)x_-\geq \sigma_-(u_--x_-^p)x_-\, .
\end{align*}
Suppose $c_-=x_-^p(0)<u_-$, then $\ddt x_-(0)>0$, and $x_-^p(t)\geq c_-$ for all $t\geq 0$.

Now suppose $c_-=u_-\leq x_-^p(0)$. Then, if for some $T\geq0$, $x_-^p(T)=c_-$, we have 
\begin{align}
    \ddt x_-(T)=\sum_{j=1}^N M_{-j}(x_j(T)-x_-(T))\geq 0.
\end{align}
Therefore $x_-^p(t)\geq c_-$ for all time. An analogous argument gives $x_+^p(t)\leq c_+$ as well.
\end{proof}

Now, any steady state of the system satisfies
\begin{equation}\label{steady}
\sum_{j=1}^N M_{ij}(x_j-x_i)+\sigma_i(u_i-x_i^p)x_i=0
\end{equation}
for each $i=1,...,N$. An immediate estimate on the steady state, from dropping the sum, gives:
\begin{equation}\label{ezb}
\min_i u_i \leq x_j^p\leq \max_i u_i,
\end{equation}
which, for identical conviction values, $u_i \equiv u$, gives a unique steady state $x_i^p=u$, for all $i=1,...,N$, which of course is a consensus state. Indeed, due to the connectivity of the graph, and along with Lemma \ref{mmp}, this further implies that unless $u_i \equiv u$ for all $i$, we have $x_i(t)>c_-$ for all $t>0$, and analogously for $x_i(t)<c_+$.\\

\subsection{Uniqueness}
Now let us show that there is a unique steady state for the model. Let us define the map $F:\mathbb R^N\to\mathbb R^N$ that assigns to any $\bx \in \mathbb R^N_+$ the vector $F(\bx)$ with entries
\begin{align}\label{Fmap}
F(\bx)_i = d_ix_i-\sum_{j=1}^NM_{ij}x_j+\sigma_i(x_i^p-u_i)x_i\, ,
\end{align}
with $d_i=\deg(i)$. Indeed, when the map $F(\bx)_i=0$ for all $i=1,...,N$, the original system is at a steady state, i.e. \eqref{steady} is satisfied.

We show that any steady state of the system is not critical for the Jacobian of this map.

\begin{lemma}\label{Jac}
Let $\bx^*$ satisfy equation \eqref{steady}. For any set of parameters $(\bu,M,\boldsymbol{\s})$, such that $u_i \in \R_+$, $M$ is symmetric, entrywise nonnegative, $M_{ij}\geq 0$, and $\s_i\geq 0$, then the Jacobian of $F$ at $\bx^*$ is not critical.
\end{lemma}
\begin{proof}
 We compute the Jacobian Matrix of $F$ as,
\begin{align}
D_{\bx}F(\bx)=G-M
\end{align}
where $G=\mathrm{diag}\{g_i\}_{i=1}^N$ and $g_i=d_i+\sigma_i((p+1)x_i^p-u_i)$. In this way we can see exactly that the Jacobian is a diagonal shift of the graph Laplacian, with diagonal shift given by $\sigma_i((p+1)x_i^p-u_i)$. Indeed if $\sigma_i((p+1)x_i^p-u_i)$ were strictly positive for all $i=1,...,N$, the shift would be stable; however, in general this is not true. Therefore, we utilize \eqref{steady} and the fact that $x_i>0$ to see that
\begin{align}
g_i=\sigma_ipx_i^p+\sum_{j=1}^NM_{ij}\frac{x_j}{x_i}.
\end{align}
Note that each $g_i>0$ while the off-diagonal entries are given by $-M_{ij}\leq 0$. Indeed the matrix for which we must compute the determinant is of the form
\[G-M= \begin{bmatrix}
g_1 & -M_{12} & \dots & -M_{1N}\\
-M_{12} & g_2 &  & \vdots \\
\vdots &  & \ddots & \vdots \\
-M_{1N} & \dots & \dots & g_N\\
\end{bmatrix}=\begin{bmatrix}
c_{11} & c_{12} & \dots & c_{1N}\\
c_{12} & c_{22} &  & \vdots \\
\vdots &  & \ddots & \vdots \\
c_{1N} & \dots & \dots & c_{NN}\\
\end{bmatrix}.
\]
Now, the determinant of $D_{\bx}F(\bx)=G-M$ is given by summing over all cycles of the graph (up to the sign of each cycle). Thus, we can explicitly compute the determinant, 
\begin{align}
\det{D_{\bx}F(\bx)}=\sum_{\d \in S_N}\mathrm{sgn}(\d) \prod_{i=1}^Nc_{i,\d(i)}
\end{align}
where  $\d$ is a cycle of the matrix $G-M$, $S_n$ is the permutation group, and $c_{i,j}$ is the element in the $i$th row and $j$th column of $D_{\bx}F(\bx)$. In this particular case, as all the off-diagonal terms are nonpositive, and the main diagonal is strictly positive, the only positive summand is the one corresponding to the main diagonal, while the rest will all be negative. Therefore, 
\begin{align}
\det{D_{\bx}F(\bx)}=\prod_{i=1}^Nc_{ii}-\sum_{\substack{\d \in S_N,\\\d \neq \{1,...,N\}}}\mathrm{sgn}(\d)\prod_{k=1}^Nc_{k,\d(k)},
\end{align}
where $-\mathrm{sgn}(\d)\prod_{k=1}^Nc_{k,\d(k)}\leq 0$ for all $\d \neq \{1,...,N\} \in S_N$.

To see that this is positive let us observe that since $c_{ii}=g_i$,
\begin{align}
\prod_{i=1}^N c_{ii}&=\prod_{i=1}^N(\sigma_ipx_i^p+\sum_{k=1}^NM_{ik}\frac{x_k}{x_i}),\\
&=p^N\prod_{i=1}^N\sigma_ix_i^p+p^{N-1}\prod_{l=1}^N\prod_{i\neq j}\sigma_ix_i^p\sum_{j=1}^NM_{jl}\frac{x_l}{x_j}+ \dots\\
&\dots +p\sum_{i=1}^N\sigma_ix_i^p\prod_{j\neq i}\sum_{l=1}^NM_{jl}\frac{x_l}{x_j}+\prod_{j=1}^N\sum_{l=1}^NM_{jl}\frac{x_l}{x_j},\\
&=I_0+\dots + I_N.\\
I_{N-n}&=p^n\sum_{i_1<i_2<...<i_n}^N\prod_{a \in \{i_n\} }\sigma_ax_a^p\prod_{l \not\in \{i_n\}}\sum_{j=1}^NM_{jl}\frac{x_l}{x_j},
\end{align}
where $\{i_n\}$ is the set of $n$ indices that are multiplying the first terms in the product $p\s_ix_i^p$, and the rest are multiplying the second terms $\sum_{k=1}^N  M_{ik}\frac{x_k}{x_i}$. Each summand of $I_j$ is positive, but in particular we see that for any pair $i,j$ the product $M_{ij}\frac{x_j}{x_i}\times M_{ji}\frac{x_i}{x_j}=M_{ij}M_{ji}.$ 

Therefore, within the summands of $I_N$ we find every cycle of the graph $M$, all with a positive sign. In fact, we can break $I_N$ down into $N$ different parts determined by the length of cycles in the graph.
\begin{align}
    I_N&=\sum_{\d \in S_N}\prod_{k=1}^NM_{k,\d(k)}+ \sum_{j=1}^N\sum_{\d \in S_{N-1}\backslash j}\left( \sum_{l\neq j}M_{jl}\frac{x_l}{x_j}\right)\prod_{k\neq j}M_{k,\d(k)}+\dots,\\
    &:=I_{N,0}+I_{N,1}+\dots,
\end{align}
where $S_{N-1}\backslash j$ denotes the permutation group of the indices $\{1,...,j-1,j+1,...,N\}$. Then in general we have for $j=1,...,N-2$
\begin{align}
    I_{N,j}&=\sum_{k_1<...<k_{j}}^N\sum_{\d \in S_{N-j}\backslash \{k_j\}}\prod_{k \in \{k_j\}}\left(\sum_{l\not\in\{k_j\}}M_{kl}\frac{x_l}{x_k}\right)\prod_{n \not\in \{k_j\}}M_{n,\d(n)},\\
    I_{N,N-1}&=\sum_{i=1}^NM_{ii}\prod_{l\neq i}M_{li}\frac{x_i}{x_l},
\end{align}
where again $S_{N-j}\backslash\{k_j\}$ is the permutation group of $N-j$ elements ignoring the $j$ indices within the set $\{k_j\}$. Thus we have decomposed $I_N$ as,
\begin{align}
    I_N=\sum_{j=0}^{N-1}I_{N,j}.
\end{align}
We do the same for each $I_{N-n}$ so that
\begin{align}
    I_{N-n}&=\sum_{j=0}^{N-n-1}I_{N-n,j},\\
    I_{N-n,0}&=p^n\sum_{i_1<...<i_n}^N\prod_{a \in \{i_n\}}\sigma_ax_a^p\sum_{\d \in S_{N-n}\backslash \{i_n\}}\prod_{k \not \in \{i_n\}}M_{k,\d(k)},\\
    I_{N-n,j}&=p^n\sum_{i_1<i_2<...<i_n}^N\prod_{a \in \{i_n\}}\sigma_ax_a^p\sum_{\substack{k_1<k_2<...<k_j \not\in \{i_n\} \\ \d \in S_{N-n-j}\backslash \{\{i_n\},\{k_j\}\}}}\prod_{l \in \{k_j\}}\sum_{m \not\in \{k_j\}}M_{lm}\frac{x_m}{x_l}\prod_{k \not \in \{\{i_n\},\{k_j\}\}}M_{k,\d(k)}.
\end{align}
Now, let us  break down the negative terms similarly,
\begin{align}
    &\sum_{\substack{\d \in S_N,\\\d \neq \{1,...,N\}}}\mathrm{sgn}(\d)\prod_{k=1}^Nc_{k,\d(k)}=\sum_{j=0}^{N-2}J_j,\\
    J_j&=\sum_{\substack{\d \in S_N \\ \exists_j k: \d(k)=k}}\prod_{k=1}^Nc_{k,\d(k)},\\
    &=\sum_{k_1<k_2<...<k_j}\sum_{\d \in S_{N-j}\backslash\{k_j\}}\prod_{k\in \{k_j\}}m_k\prod_{l\not\in\{k_j\}}M_{l,\d(l)}.
\end{align}
Now, as $m_k=p\sigma_kx_k^p+\sum_{j=1}^NM_{kj}\frac{x_j}{x_k}$, we can further split up each $J_j$ into parts determined by how many powers of $p$ the part has,
\begin{align}
    J_j&=\sum_{k=0}^{j}J_{j,k},\\
    J_{j,k}&=p^k\sum_{\substack{k_1<...<k_j \\ \d \in S_{N-j}\backslash \{k_j\}}}\prod_{l \not\in \{k_j\}} M_{l,\d(l)}\sum_{n_1<...<n_k}\prod_{n\in \{n_k\}}\sigma_nx_n^p\prod_{\substack{m\not\in \{n_k\} \\  m \in \{k_j\}}}\sum_{s\not\in \{k_j\}}M_{ms}\frac{x_s}{x_m}.
\end{align}

Finally, by combining all parts that have the same powers of $p$, from the positive $I$ terms, and the negative $J$ terms we have
\begin{align}
    I_{N-n,k}=J_{n+k,n}, \ \ 0\leq k\leq N-n-2,
\end{align}
where all positive $p^n$ powers are found in $I_{N-n}=\sum_{k=0}^{N-n-1}I_{N-n,k}$ and all negative powers within $\sum_{k=n}^{N-2}J_{k,n}$. In which case,
\begin{align}
    I_{N-n}-\sum_{k=n}^{N-2}J_{k,n}=I_{N-n,N-n-1}\geq 0.
\end{align}
Thus summing over all $n=0,...N$ we get,
\begin{align}
    \det{D_{\bx}F(\bx)}=\sum_{n=0}^N \left(I_{N-n}-\sum_{k=n}^{N-2}J_{k,n}\right)=I_0+I_1+\sum_{n=0}^{N-2}I_{N-n,N-n-1}>0,
\end{align}
Where $I_0,I_1>0$ always, and $\sum_{n=0}^{N-2}I_{N-n,N-n-1}=0$ in the case of no self loops within the underlying graph, and otherwise is strictly positive with self loops.

Noting that this computation is identical for any principal minor of the matrix, we conclude that the determinant of every principal minor is also strictly positive, and the Jacobian Matrix is in fact an invertible M-Matrix with strictly positive eigenvalues.
\end{proof}

If the payout function \eqref{eq:pNFJ} were convex, then non-degeneracy of the Jacobian would immediately grant uniqueness of the steady state. However, without convexity we utilize the Brouwer topological degree in order to achieve the result. To achieve the uniqueness we utilize the fact that on a particular region $\cW$, the topological degree of a point $\by$ is defined as
\begin{equation}
    \mathrm{deg}\{F,\cW,\by\}=\sum_{\bx \in F^{-1}(\boldsymbol{y})} \mathrm{sgn} (\det{D_{\bx}F(\bx)}).
\end{equation}
Now, since we have determined that  $\mathrm{sgn} (\det{D_{\bx}F(\bx)})=1$, for any $\bx \in F^{-1}(\boldsymbol{0})$, if we can show that the degree is exactly one, then there must be a unique fixed point within the region  $\cW$.
For background on the Brouwer topological degree see \cite{Cronin}. 

\begin{lemma}
For any set of parameters $(\bu,M,\boldsymbol{\s})$, such that $u_i \in \R_+$, $M\geq 0$ is symmetric, and $\s_i\geq 0$, there exists a unique solution $\bx^*$ to \eqref{steady} which is a locally exponentially stable equilibrium for the system \eqref{model}. The map $(\bu,M,\boldsymbol{\s})\mapsto \bx^*$ 
is infinitely smooth.
\end{lemma}

\begin{proof}
We consider the Brouwer topological degree of $F$ at zero. To define the degree properly, we restrict $F$ to a wedge region $\cW$. Let us denote
\begin{align}
\langle \bx,\by\rangle= \sum_{i=1}^N \sigma_i x_i y_i, \ \ \ \ \ \|\bx\|_p^p=\sum_{i=1}^N \sigma_ix_i^p.
\end{align}
We define,
\begin{align}
\cW=\{\bx: x_i \geq 0, \ \e\leq \|\bx\|_{\infty}, \|\bx\|_{p+1} \leq R\},
\end{align}
where $R>0$ is large, and $\e$ small, to be determined momentarily. We verify that the image of the boundary does not contain the origin, $0 \not\in F(\partial \cW)$. Indeed, if $x_i=0$ for some $i$,  then $F_i=-\bar{x}<0.$ Now we compute,
\begin{align}
\sum_{i=1}^N F(\bx)_i=-\langle \boldsymbol{u},\bx\rangle +\|\bx\|_{p+1}^{p+1}.
\end{align}
Thus, if $\|\bx\|_{p+1}=R$, we have the bound
\begin{align}
\sum_{i=1}^N F(\bx)_i\geq \|\bx\|_{p+1}^{p+1}-\|\bx\|_{p+1}\|\boldsymbol{u}\|_{\frac{p+1}{p}}>0,
\end{align}
provided $R$ is large enough. Similarly, if $\|\bx\|_{\infty}=\e$, then
\begin{align}
\sum_{i=1}^N F(\bx)_i \leq -u_-\|\bx\|_1+\e^p\|\bx\|_1<0,
\end{align}
provided $\e$ is small enough.

Therefore the value $\boldsymbol{0}$ of $F$ is regular, and its degree can be computed explicitly by
\begin{align}
\mathrm{deg}\{F,\cW,\boldsymbol{0}\}=\sum_{\bx \in F^{-1}(\boldsymbol{0})} \mathrm{sgn} (\det{D_{\bx}F(\bx)}).
\end{align}
However, we have proved that all Jacobians for $\bx \in F^{-1}(\boldsymbol{0})$ are strictly positive. Therefore uniqueness can be shown by proving $\mathrm{deg}\{F,\cW,\boldsymbol{0}\}=1$. This is certainly true for $\hat{\boldsymbol{u}}=(u,\dots,u)$, since we have a unique positively oriented solution from \eqref{ezb}. Indeed, this is because for identical conviction values, consensus is achieved at exactly the conviction value, $u^{1/p}$. Now, fix any such $\hat{\boldsymbol{u}}$ and consider the homotopy of maps
\begin{align}
F^{(\tau)}:=F_{\tau \bu + (1-\tau)\hat{\bu}} 
\end{align}
where $F_{\by}$ denotes the map $F$ in \eqref{Fmap} with $u$ replaced by $\by$. 
Since $\boldsymbol{0} \not\in F^{(\tau)}(\partial \cW)$ for any $\tau$, the Invariance under Homotopy Principle applies and hence,
\begin{equation}
\mathrm{deg}\{F_{\boldsymbol{u}},\cW,\boldsymbol{0}\}=\mathrm{deg}\{F_{\hat{\boldsymbol{u}}},\cW,\boldsymbol{0}\}=1,
\end{equation}
and the proof of uniqueness is finished.

The smoothness of the steady state, $\bx^*$, as a function of $(\boldsymbol{u},M, \boldsymbol{\sigma})$ follows directly from the non-degeneracy of the Jacobian and the Implicit Function Theorem.
\end{proof}

\subsection{Convergence}
With this lemma in hand we need only prove that for any initial data the opinion values actually converge to this equilibrium, achieving the Nash Equilibrium.

Convergence to the equilibrium is proved by first revealing the gradient structure for the system \eqref{model}.
\begin{equation}\label{eq:GF}
\ddt \bx=-\nabla \Phi(\bx),
\end{equation}
for
\begin{align}
\Phi(\bx)=\frac{1}{4}\sum_{i,j=1}^NM_{ij}(x_i-x_j)^2+\frac{1}{p+2}\sum_{i=1}^N \sigma_ix_i^{p+2}-\frac{1}{2}\sum_{i=1}^N\sigma_iu_ix_i^2.
\end{align}
We note that as $\Phi(\bx)$ is bounded, both from above and below, and decreases along the flow, we know that $\Phi(\bx(t))$ must converge somewhere. 

To prove the convergence of $\bx(t)$ itself we will appeal to the Lojasiewicz gradient inequality, stated next.
\begin{theorem}{\cite{L}}
Let $\Phi$ be a real analytic function in a neighborhood of $U$. Then for any $\bx_0 \in U$, there are constants $c>0$,  $\d\in(0,1]$, and $\mu \in [1/2,1)$, such that
\begin{align}
    \|\nabla\Phi(\bx)\|\geq c|\Phi(\bx)-\Phi(\bx_0)|^{\mu}, \ \ \ \ \forall \bx \in U, \ \ \text{such} \ \text{that} \ \|\bx-\bx_0\|\leq \delta.
\end{align}
\end{theorem}
As our $\Phi$ is real analytic for all $x_i \in \R_+$, the result applies. Consider a solution $\bx$ in the positive sector, $\R_+^N$ to \eqref{eq:GF}. As every solution is bounded there must be an accumulation point $\bx_0$. We now must show that $\bx(t) \to \bx_0$, and that $\nabla \Phi(\bx_0)=0$, establishing that $\bx_0=\bx^*$, the unique steady state. First, we need a lemma to control the length of the orbit $\bx(t)$ near the accumulation point $\bx_0$.

\begin{lemma}\label{arc}
As long as $\bx(t) \in B_{\d}(\bx_0)$ for $t'\leq t\leq t''$, we have,
\begin{align}
    \int_{t'}^{t''}\|\dot{\bx}(s)\|\ds\leq \int_{\Phi(\bx(t''))}^{\Phi(\bx(t'))} \frac{1}{c|\zeta-\Phi(\bx_0)|^{\mu}}d\zeta.
\end{align}
\end{lemma}
\begin{proof}
Let us define the two functions,
\begin{align}
    \psi(\zeta)&=c|\zeta-\Phi(\bx_0)|^{\mu},\\
    \Psi(x)&=\int_0^x\frac{1}{\psi(\zeta)}d\zeta.
\end{align}
Let us compute,
\begin{align}
    \ddt \Psi(\Phi(\bx(t)))=\dot{\Psi}(\Phi))\dot{\Phi}=\frac{\dot{\Phi}}{\psi(\Phi)}
\end{align}
but as $\dot{\Phi}(\bx(t))=-\|\nabla\Phi\|^2$, we get,
\begin{align}
    -\ddt \Psi(\Phi(\bx(t)))=\frac{\|\nabla \Phi\|^2}{c|\Phi(\bx(t))-\Phi(\bx_0)|^\mu}\geq \|\dot{\bx}\|,
\end{align}
by the Lojasiewicz inequality.
Integrating over $(t',t'')$ proves the lemma.
\end{proof}
With this lemma in hand, we are ready to prove convergence to the equilibrium.
\begin{proof}[Proof of Theorem \ref{T:main}]
As $\bx_0$ is an accumulation point, there is a sequence $t_n$ such that $\bx(t_n) \to \bx_0$. Fix an arbitrary $r<\d$ and choose a remote time $t_n$ far enough so that
\begin{align}
    \|\bx(t_n)-\bx_0\|<\frac{r}{2}, \ \ \ \int_{\Phi(\bx_0)}^{\Phi(\bx(t_n))}\psi(\zeta) d\zeta<\frac{r}{2}.
\end{align}
Now, to show that the entire trajectory for $t>t_n$ remains in $B_r(\bx_0)=\{\by:\|\bx_0-\by\|\leq r\}$, let $\tilde{t}$ be the first time such that $\|\bx(t_n+\tilde{t})-\bx_0\|=r$. Then $\bx(t)$ lies in $B_r(\bx_0)$ on $(t_n,t_n+\tilde{t})$. Then applying Lemma \ref{arc} we get,
\begin{align}
    \|\bx(t_n+\tilde{t})-\bx_0\|\leq \|\bx(t_n+\tilde{t})-\bx(t_n)\|+\|\bx(t_n)-\bx_0\|<r.
\end{align}
Therefore the trajectory must lie in $B_r(\bx_0)$ for all $t>t_n$. To conclude that $\bx_0=\bx^*$, we note that the above argument implies
\begin{align}
    \int_{t_n}^{\infty} \|\dot{\bx}(s)\|\ds<\infty.
\end{align}
Therefore $\dot{\bx}(s_n)\to 0$ and thus $\nabla \Phi(\bx(s_n))\to 0=\nabla\Phi(\bx_0)$.
\end{proof}

\section{Continuous Friedkin-Johnsen: Taylor Model}\label{s:CFJ}
The Taylor model \eqref{eq:FJc} is also representative of a gradient flow, and further its payout function is convex, so the standard convexity theory guarantees existence, uniqueness and convergence to the Nash Equilibrium. However, the proof of Theorem \ref{T:main} is powerful as it also applies as an alternate proof for the Taylor Model \eqref{eq:FJc}. Indeed, steady states of \eqref{eq:FJc} satisfy
\begin{align}\label{steady:FJ}
    \l_i\sum_{j=1}^NA_{ij}(x_j-x_i)+(1-\l_i)(u_i-x_i)=0,
\end{align}
and therefore have corresponding $F$-maps,
\begin{align}
    F(\bx)_i=\l_ix_i-\l_i\sum_{j=1}^NA_{ij}x_j+(1-\l_i)(x_i-u_i).
\end{align}
The Jacobian Matrix of this map is given by
\begin{align}
    D_{\bx}F(\bx)=I-\L A,
\end{align}
where the elements of $\L A$ are given by $\l_iA_{ij}$. By the Perron-Frobenius theory, the determinant of this Jacobian must be positive unless $\l_i=1$ for all $i$, which of course reduces the model to $\eqref{eq:EAc}$. The proof then follows exactly as for the nonlinear model, noting that the gradient structure needed for convergence can be found by the same change of variable used to transform \eqref{eq:FJ} to \eqref{eq:FJs}. If we let $y_i=\frac{x_i}{\l_i}$ and $\s_i=\frac{1-\l_i}{\l_i}$ then  \eqref{eq:FJc} becomes
\begin{align}
   \dot{y}_i= \sum_{j=1}^NB_{ij}(y_j-y_i)+\s_iu_i+(\l_i-1)y_i.
\end{align}
Then, the gradient structure can be revealed,
\begin{align}
    \ddt\by=-\nabla \Phi(\by),
\end{align}
for
\begin{align}
    \Phi(\by)=\frac{1}{4}\sum_{i,j=1}^NB_{ij}(y_i-y_j)^2+\sum_{k=1}^N (\s_ku_ky_k+ \frac{1}{2}(\l_k-1)y_k^2).
\end{align}

Therefore, the Taylor model must also converge to a unique Nash Equilibrium. In the next section, we present some numerical experiments which compare the nonlinear model with the original Friedkin-Johnsen model.

\section{Numerical experiments}\label{sec:experiments}
In this section, we present a few illustrative examples to show the different behavior of the opinion dynamics modeled by the linear FJ model and  he nonlinear model introduced here, when subject to the same initial convictions and same stubbornness parameters. 
When opinion values are close to $1$, the  nonlinear model \eqref{eq:NFJc} behaves  similarly to the linear version \eqref{eq:FJc}, as when  $x_i \sim 1$,  the linear term $u_i-x_i$  well-approximates the nonlinear one $(u_i-x_i^p)x_i$. However, we see that the nonlinear model becomes more stubborn for large opinion values than the linear version.  Indeed, the lack of translation invariance leads exactly to this behavior. The behavior represented by the nonlinear model is that agents which hold an extreme opinion, even with low stubbornness parameter $\sigma \ll 1$, require much more external pressure to change their minds, and will anyways continue to hold extreme opinions.
\begin{figure}[t]
    \centering
    \includegraphics[width=\textwidth]{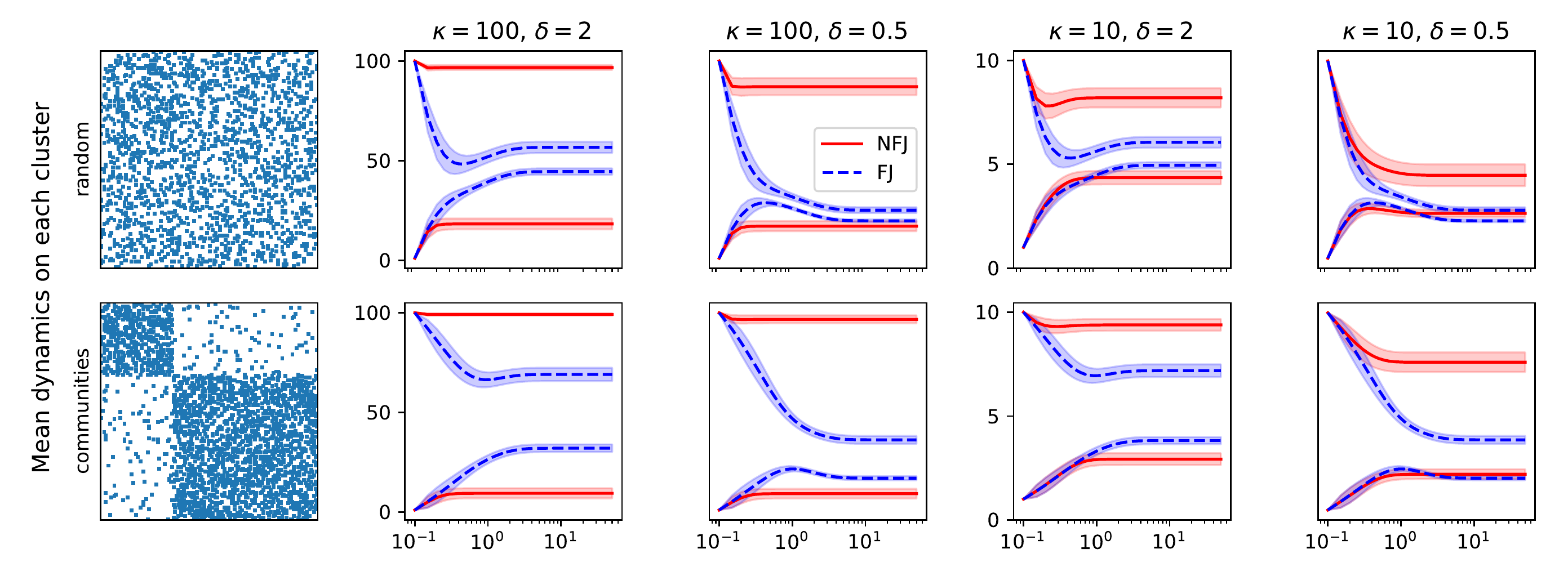}
    \caption{Adjacency matrix spy plot, and mean and standard deviation of node dynamics $x_i(t)$ in the time window $t\in [0,100]$ for the FJ  \eqref{eq:experiments_linear_model} and the NFJ \eqref{eq:experiments_nonlinear_model} models, for different choices of $u_i$ and $\sigma_i$, as defined in \eqref{eq:experiments_u_sigma}. Top row: Erdős–Rényi  random graph with 150 nodes and $p_e=0.08$ edge probability; Bottom row: SBM random graph with two blocks of 50 and 100 nodes, respectively, within-block edge probability $p_{in}=0.2$ and between-block edge probability $p_{out}=0.02$.}
    \label{fig:random_vs_communities}
\end{figure}
This is illustrated in Figure \ref{fig:random_vs_communities} where we present an example of the opinion dynamics of the two models on a small random graph and a strongly community-structured random network. The random network is an Erdős–Rényi graph with $150$ nodes and  $p_e=0.08$ edge probability. The community-structured network is a stochastic block model with two clusters of size $50$ and $100$ with within-cluster probability $p_{in}=0.2$ and across-cluster edge probability $p_{out}=0.02$.  
%
%
For a fair comparison of the choice of the parameters $\sigma_i$ and the initial convictions $u_i$, the linear model considered here is the following modified FJ or Taylor model 
\begin{equation}\label{eq:experiments_linear_model}
    \dot x_i(t) = \sum_{j=1}^N M_{ij}(x_j(t)-x_i(t))+\sigma_i(u_i-x_i(t))
\end{equation}
which we compare with the considered nonlinear dynamical system \eqref{eq:NFJc} for $p=1$
\begin{equation}\label{eq:experiments_nonlinear_model}
    \dot x_i(t) = \sum_{j=1}^N M_{ij}(x_j(t)-x_i(t))+\sigma_i(u_i-x_i(t))x_i(t)\, .
\end{equation}
For both the network configurations, the conviction and the stubbornness vectors are set to 
\begin{equation}\label{eq:experiments_u_sigma}
    u_i = \begin{cases}
\kappa & 1\leq i\leq 50\\
1       & 51\leq i\leq 150    
\end{cases}\qquad \text{and}\qquad 
\sigma_i = \begin{cases}
\delta & 1\leq i\leq 50\\
1       & 51\leq i\leq 150    
\end{cases}
\end{equation}
and we evaluate the dynamics $x_i(t)$ of each node $i$ from $t=0$ to $t=100$, for $\kappa\in\{10,100\}$ and $\delta\in\{1/2,2\}$. 

As the linear and nonlinear FJ models \eqref{eq:FJc} and \eqref{eq:NFJc}  attempt to push each opinion towards the conviction values $u_i$ and $u_i^{1/p}$, respectively, the choice of $p=1$ allows us to better appreciate the comparison between the two models with respect to the same choice of convictions. 
Although other values of $p$ would accentuate the system's nonlinear behavior for values of $x_i$ away from $1$, the choice $p=1$ introduces enough nonlinearity to significantly change the overall dynamics, and the qualitative difference between the two models does not change for different values of $p$.

\begin{figure}[t]
    \centering
    \includegraphics[width=\textwidth]{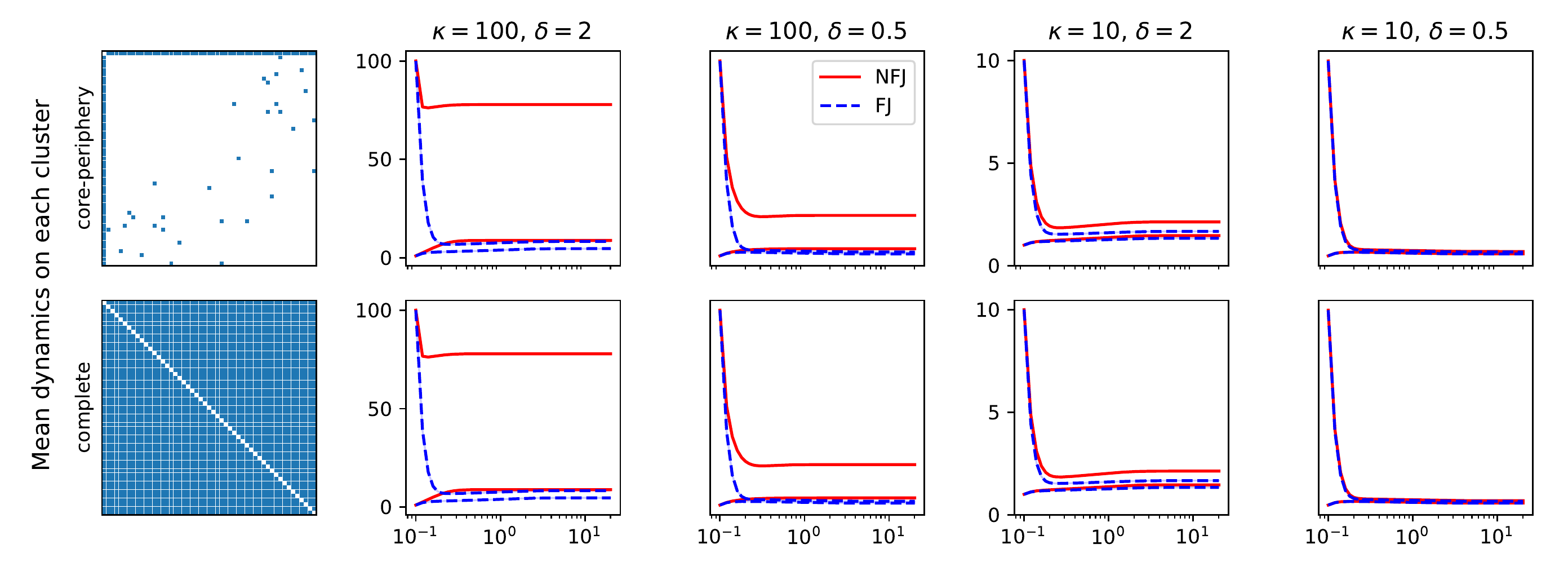}
    \caption{Adjacency matrix spy plot, and mean node dynamics $x_i(t)$ in the time window $t\in [0,100]$ for the FJ  \eqref{eq:experiments_linear_model} and the NFJ \eqref{eq:experiments_nonlinear_model} models, for different choices of $u_i$ and $\sigma_i$, as defined in \eqref{eq:experiments_u_sigma}. Top row: strongly-structured core-periphery network with one single core node $i=1$, such that $(1,j)\in\mathcal E$ for all $j\in \mathcal V$, while edges $ij$ with $i,j\neq 1 $ exist with small probability $p_e=0.01$; Bottom row: complete (loop-free) graph, with $ij\in \mathcal E$ for all $i,j\in \mathcal V$.}
    \label{fig:cp_vs_full}
\end{figure}
Figure \ref{fig:cp_vs_full} illustrates the fact that both the FJ and the NFJ models with complete all-to-all coupling can be well approximated by a core-periphery network with a  drastically reduced edge-density, where only one single node $i=1$ is connected to every other node, while edges $ij$ with $i,j\neq 1 $ exist with small probability $p_e=0.01$. Indeed, the behavior seen in the first row is identical to that in the second row, while still highlighting the differences in stubbornness between the FJ and NFJ models. Further, we see that both FJ and NFJ approach consensus in the final two columns of Figure \ref{fig:cp_vs_full} as the strong mixing of the networks overpowers the small convictions $\k=10$ and small stubbornness $\d=0.5$.
Finally, Figure~\ref{fig:real_datasets} illustrates the behavior of FJ and NFJ on two real-world social networks: `Jazz', a  network of Jazz bands,  consisting of 198 nodes, being
jazz bands, and 2742 edges representing common musicians; `CollegeMsg', a  network accounting for 13838 exchanged messages (the edges) among 1899 students (the nodes) in a north American college. In this case, similarly to what was done in \eqref{eq:experiments_u_sigma}, we defined $u$ and $\sigma$ by randomly assigning $N/2$ entries the value $1$ and the other $N/2$ the value $\kappa$ and $\delta$, respectively. As expected, the behavior of the dynamics resembles the one observed on previous synthetic sparse graphs.  

\begin{figure}
    \centering
    \includegraphics[width=\textwidth]{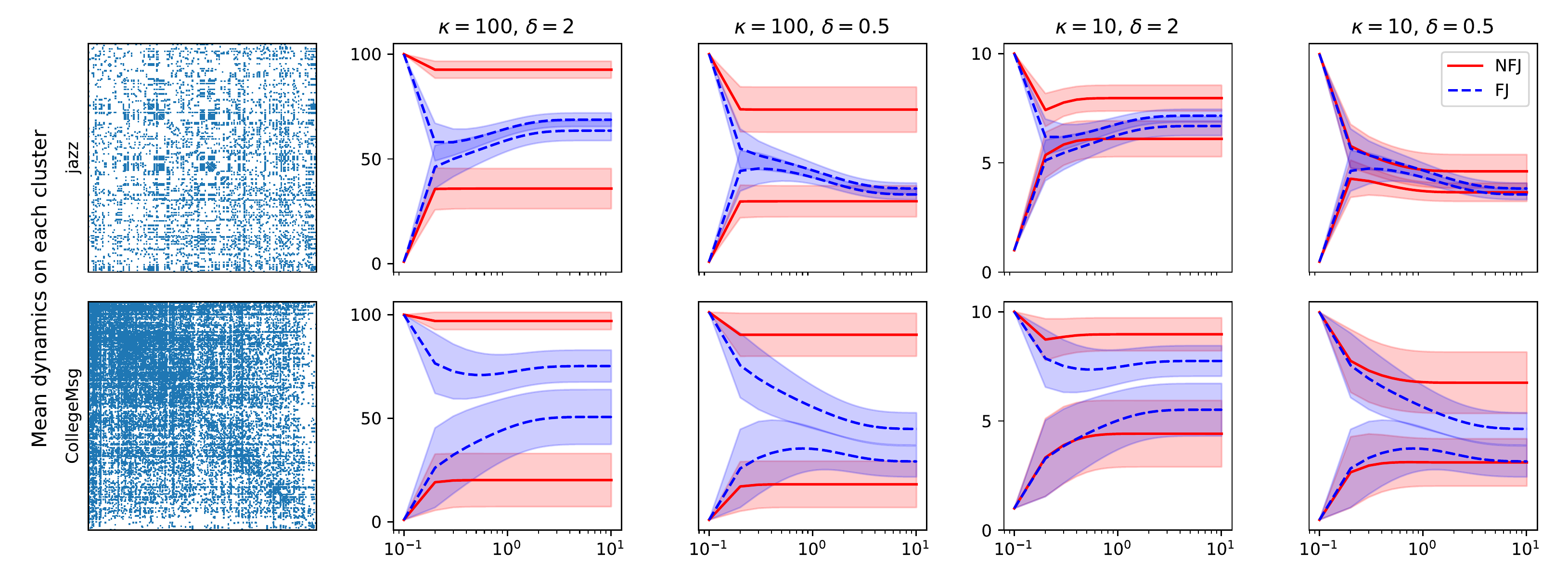}
    \caption{Dynamics on real world networks. Adjacency matrix spy plot, and mean node dynamics $x_i(t)$ in the time window $t\in [0,100]$ for the FJ  \eqref{eq:experiments_linear_model} and the NFJ \eqref{eq:experiments_nonlinear_model} models, for different choices of $u_i$ and $\sigma_i$, as defined in \eqref{eq:experiments_u_sigma}. Top row: `Jazz' network consisting of 198 nodes, jazz bands, and 2742 edges representing common musicians; Bottom row: `CollegeMsg' network containing 1899 nodes, students, and 13838 edges representing messages sent between students.}
    \label{fig:real_datasets}
\end{figure}

\section{Discussion}
In this work we introduced a nonlinear update protocol for modelling opinion dynamics over networks. The nonlinearity was inspired by the Rayleigh friction and self-propulsion, used in both flocking and oscillatory models like the Cucker-Smale and Stuart-Landau systems. We prove well-posedness and convergence of the model to a unique Nash Equilibrium. The techniques involved directly computing the Jacobian of the steady states, utilizing the Brouwer topological degree, and analysis of a nonconvex gradient flow. The techniques used for the nonlinear model carry over directly to the existing literature of linear models that are also discussed in this article. Further, it expands the nonlinear theory to more diverse network topologies as well as heterogeneous stubbornness parameters. Finally, we provide several numerical experiments to highlight the similarities and differences in outcomes that arise from utilizing nonlinear effects in the modeling of opinion dynamics. Our model effectively describes the phenomena: Extreme opinions are less easily swayed by social pressure.

The model presented here has a strong connection with models for coupled Stuart-Landau models that will be the topic of a future study. Future directions within opinion dynamics include generalizing what kind of nonlinear forcings and consensus mechanisms are permissible to lead to the same sort of Nash equilibrium.


\end{document}